\documentclass[11pt,a4paper]{amsart}

\usepackage{url}
\usepackage[colorlinks,linkcolor=blue,citecolor=blue]{hyperref}

\date{To appear in \textit{Mathematics of Computation} (accepted, 2011-05-02)}

\usepackage{amsmath,amsthm,amsfonts}

\newcommand{\N}{\mathbb{N}}
\newcommand{\Z}{\mathbb{Z}}
\newcommand{\Q}{\mathbb{Q}}
\newcommand{\R}{\mathbb{R}}
\newcommand{\C}{\mathbb{C}}

\newcommand{\B}{\mathcal{B}}
\newcommand{\E}{\mathcal{E}}

\renewcommand{\Re}{\operatorname{Re}}
\renewcommand{\Im}{\operatorname{Im}}

\newcommand{\dist}{\operatorname{dist}}

\newcommand{\res}{\operatorname{Res}}
\newcommand{\eqdef}{\overset{\mathrm{def}}{=}}

\theoremstyle{plain}
\newtheorem{theorem}{Theorem}

\newtheorem{corollary}[theorem]{Corollary}

\newtheorem{prop}[theorem]{Proposition}

\newtheorem{lemma}{Lemma}

\theoremstyle{remark}
\newtheorem{remark}{Remark}

\newtheorem{example}{Example}

\title[Estimates for Apostol-Bernoulli and Apostol-Euler polynomials]%
{Asymptotic estimates for Apostol-Bernoulli and Apostol-Euler polynomials} 

\author[L.~M.~Navas]{Luis M. Navas}
\address{%
        Departamento de Matem\'aticas,
        Universidad de Salamanca,
        Plaza de la Merced 1-4,
        37008 Salamanca, Spain}
\email{navas@usal.es}

\author[F.~J.~Ruiz]{Francisco J. Ruiz}
\address{%
        Departamento de Matem\'aticas,
        Universidad de Zaragoza,
        Campus de la Plaza de San Francisco,
        50009 Zaragoza, Spain}
\email{fjruiz@unizar.es}

\author[J.~L.~Varona]{Juan L. Varona}
\address{%
        Departamento de Matem\'aticas y Computaci\'on,
        Universidad de La Rioja,
        Calle Luis de Ulloa s/n,
        26004 Logro\~no, Spain}
\email{jvarona@unirioja.es}
\urladdr{http://www.unirioja.es/cu/jvarona/welcome.html}

\thanks{Research of the second and third authors supported by grant MTM2009-12740-C03-03 of the~DGI}

\keywords{Apostol-Bernoulli polynomials, Apostol-Euler polynomials, 
Fourier series, asymptotic estimates}

\subjclass[2010]{Primary 11B68; Secondary 42A10, 41A60}

\begin{document}

\maketitle

\begin{abstract}
	We analyze the asymptotic behavior of the Apostol-Bernoulli polynomials
	$\B_{n}(x;\lambda)$ in detail. The starting point is their Fourier series on $[0,1]$
	which, it is shown, remains valid as an asymptotic expansion 
	over compact subsets of the complex plane.
	This is used to determine explicit estimates on the constants in the approximation, and 
	also to analyze oscillatory phenomena which arise in certain cases.

	These results are transferred to the Apostol-Euler polynomials $\E_{n}(x;\lambda)$
	via a simple relation linking them to the Apostol-Bernoulli polynomials.
\end{abstract}

\section{Introduction}\label{sec:intro}

The family of Apostol-Bernoulli polynomials $\B_{n}(x;\lambda)$ in the variable $x$ and parameter $\lambda \in \C$ were introduced by Apostol in~\cite{Ap}, where they are defined by means of the power series expansion at $0$ of the meromorphic generating function
\begin{equation}
\label{eq:GF-AB}
  g(x,\lambda,z) \stackrel{\text{\tiny def}}{=} \frac{z e^{xz}}{\lambda e^z-1} = \sum_{n=0}^\infty \B_n(x;\lambda)\, \frac{z^n}{n!}.
\end{equation}
The set $S$ of poles of $g$ is $\{ 2\pi i n - \log\lambda : n \in \Z\}$ when $\lambda \neq 1$ and $\{ 2 \pi i n : n \in \Z, n \neq 0 \}$ when $\lambda = 1$, due to the fact that $0$ is a removable singularity in the latter case. Here and throughout this paper, $\log$ is the principal branch of the logarithm; namely, for $\lambda \neq 0$,\ $\log\lambda = \log |\lambda| + i \arg \lambda$, with $-\pi < \arg \lambda \le \pi$; in particular $\log 1 = 0$. 
This change in the set of poles is reflected in various discontinuities as $\lambda \to 1$; for example, the radius of convergence of the series in \eqref{eq:GF-AB} is $2\pi$ when $\lambda = 1$ and $|\log \lambda|$ when $\lambda \neq 1$.

The value $\lambda = 1$ corresponds to the classical Bernoulli polynomials, i.e.\ $\B_n(x;1) = B_n(x)$, but it is certainly not the case that $B_{n}(x) = \lim_{\lambda \to 1}\B_{n}(x;\lambda)$.
There \emph{is} a limiting relationship between $\B_{n}(x;\lambda)$ and $B_{n}(x)$ as $\lambda \to 1$, but it is not immediately obvious. Another aspect of this discontinuity is that, although $B_{n}(x)$ is monic of degree $n$, for $\lambda \neq 1$ the degree of $\B_{n}(x;\lambda)$ is $n-1$ and its leading term is $n/(\lambda-1)$.

We follow the convention of denoting $\B_{n}(\lambda) = \B_{n}(0;\lambda)$; these are the Apostol-Bernoulli \emph{numbers}. In fact, $\B_{n}(\lambda)$ is a rational function in $\lambda$ with denominator $(\lambda-1)^{n+1}$ and
coefficients related to the Stirling numbers. The classical Bernoulli numbers are given by $B_{n} = \B_{n}(1)$. The above remarks concerning the discontinuity at $\lambda = 1$ also apply here.

The case $\lambda = 0$ is trivial; indeed $\B_0(x;0) = 0$ and $\B_n(x;0) = -nx^{n-1}$ for $n \ge 1$. In particular 
$\B_{1}(0) = -1$ and $\B_{n}(0) = 0$ for $n \neq 1$. For this reason we will assume $\lambda \neq 0$ in what follows.

Dilcher showed in~\cite{Dil}, using properties of the Riemann zeta function, that the Bernoulli polynomials satisfy
\begin{equation}
\label{eq:lim-B}
\begin{aligned}
  \lim_{n \to \infty} \frac{(-1)^{n-1} (2\pi)^{2n}}{2 (2n)!} \,B_{2n}(z)
  &= \cos(2\pi z),
\\
  \lim_{n \to \infty} \frac{(-1)^{n-1} (2\pi)^{2n+1}}{2 (2n+1)!} \,B_{2n+1}(z)
  &= \sin(2\pi z),
\end{aligned}
\end{equation}
uniformly on compact subsets of~$\C$. In addition, the difference between the $n$th term and its limit is found to be of the order $O(2^{-n})$, with the implicit constant depending exponentially on~$|z|$.
The authors showed in~\cite{NRV-JAT} how these facts also follow easily, at least on $[0,1]$, from the Fourier expansion of the Bernoulli polynomials, including the quantitative bounds for the differences, and also precise estimates on the rate of convergence, namely bounds for the successive quotients of these differences.

The purpose of this article is to obtain analogous asymptotic estimates for 
$\B_n(z;\lambda)$, valid for any $z \in \C$. In short, the central result of this paper is that the Fourier series
\eqref{eq:FS-AB} of $\B_{n}(x;\lambda)$ for $x \in [0,1]$, which a priori represents it only on this interval, is actually valid on the entire complex plane as an \emph{asymptotic} series representing $\B_{n}(z;\lambda)$ for $z \in \C$. From this we deduce explicit asymptotic estimates for the Apostol-Bernoulli polynomials which include the pattern mentioned above, namely, geometric order of decrease for the differences between $\B_{n}(z;\lambda)$ and its asymptotic approximations, with implicit constants that are exponential in $|z|$, as well as estimates for succesive quotients of these differences. The behavior of the approximations varies considerably depending on $\lambda$, with $\lambda = 1$, studied in~\cite{NRV-JAT}, turning out to be the exception rather than the rule.

Many authors, beginning with Apostol in the foundational paper \cite{Ap}, use transcendental methods when studying the Apostol-Bernoulli polynomials, due to their relation with the Lerch transcendent $\Phi$, defined by analytic continuation of the series
$$
	\Phi(\lambda,s,a) = \sum_{k=0}^{\infty} \frac{\lambda^{k}}{(k + a)^{s}},
$$
where $a \neq 0,-1,-2,\ldots$ and either $|\lambda| < 1, s \in \C$ or $|\lambda|=1, \Re s > 1$
guarantees convergence (we use $\lambda$ as a variable in order to maintain a unified notation).
One has 
\begin{equation}
\label{eq:AB-Phi}
	\B_{n}(a;\lambda) = -n \, \Phi(\lambda,1-n,a)
\end{equation}
on various domains of analytic continuation, and this relation is often exploited to obtain identities satisfied by $\B_{n}$ as special cases of identities satisfied by~$\Phi$. The Fourier series of $\B_{n}(x;\lambda)$ is an instance of this (see Section~\ref{sec:FourierSeriesAB}).

While certainly not denying the great merit of this approach, it can be overkill, adding unnecessary effort to the study of what is, after all, a polynomial family. 
In this paper we wish to bring to light that the \emph{algebraic} properties of the Apostol-Bernoulli polynomials and basic Fourier analysis can get us equally far in describing their asymptotic behavior.
We feel this is interesting in its own right, more so as the algebraic properties involved are precisely those of the so-called ``umbral'' variety and hence the method probably works in a wider context. For instance, this point of view yields an overlooked elementary proof of the Fourier expansion \eqref{eq:FS-AB} itself.

The same methods used to study the Apostol-Bernoulli polynomials may also be applied to the Apostol-Euler polynomials $\E_{n}(x;\lambda)$ introduced by Luo and Srivastava (see~\cite{Luo-Tai,LuoSri}). In Section~\ref{sec:ApostolEuler}, using a relation which apparently has not been previously remarked on in the literature (see Lemma~\ref{lem:relAB-AE2}), we will show that the analogous results for $\E_{n}$ are consequences of those for $\B_{n}$; hence we concentrate on the latter and summarize the results for the former.

\section{The Fourier series of $\B_{n}(x;\lambda)$}
\label{sec:FourierSeriesAB}

In \cite[Theorem 2.1,\,p.~3]{Luo-MC}, it is proved that the Fourier series of $\B_n(x;\lambda)$ for any $\lambda \in \C$ with $\lambda \neq 0$ is
\begin{equation}
\label{eq:FS-AB}
  \B_n(x;\lambda) = -\delta_n(x;\lambda) - \frac{n!}{\lambda^x} \sum_{k \in \Z \setminus\{0\}} 
  \frac{e^{2\pi i k x}}{(2\pi i k - \log \lambda)^n},
\end{equation}
where $\delta_n(x;\lambda) = 0$ or $\frac{(-1)^n n!}{\lambda^x \log^n \lambda}$ according as $\lambda = 1$ or $\lambda \ne 1$. This expansion is valid for $0 \le x \le 1$ when $n \ge 2$ and for $0 < x < 1$ when $n = 1$. It yields the Fourier series of the Bernoulli polynomials, due to Hurwitz (1890), as the special case $\lambda = 1$.

The expansion \eqref{eq:FS-AB} is proved from scratch in~\cite{Luo-MC} starting from the generating function \eqref{eq:GF-AB} and applying the Lipschitz summation formula (which can be derived in turn from the Poisson summation formula). It may also be obtained immediately via the relation~\eqref{eq:AB-Phi} to the Lerch transcendent by specializing the following series expansion for $\Phi$, that can be found in~\cite[formula 1.11\,(6), p.~28]{Bate}:
\begin{equation}
\label{eq:AS-Lerch}
	\Phi(\lambda,s,x) = \lambda^{-x} \Gamma(1-s) \sum_{k \in \Z} (2\pi i k - \log \lambda)^{s-1} e^{2 \pi i k x},
\end{equation}
for $0 < x \leq 1$, $\Re s < 0$, $\lambda \notin (-\infty,0]$.

As we have mentioned in the introduction, our approach does not require the use of the Lerch transcendent.
In this vein, we present an elementary proof of \eqref{eq:FS-AB}, using only the algebraic properties of the Apostol-Bernoulli polynomials and the real Riemann integral. The main ingredient is the fact that $\{\B_{n}(x;\lambda)\}$ is
an \emph{Appell sequence} for fixed~$\lambda$; namely, it satisfies the derivative relation
\begin{equation}
\label{eq:appell}
	\B'_{n}(x;\lambda)=n\B_{n-1}(x;\lambda),
\end{equation}
as can be easily deduced from the form of the generating function \eqref{eq:GF-AB}.
One merely needs to observe that \eqref{eq:FS-AB} is equivalent to finding the Fourier coefficients of $x \mapsto \lambda^{x}\B_n(x;\lambda)$ for fixed~$\lambda$. 

\begin{prop}
\label{prop:FC-AB}
For any $\lambda \in \C$, $\lambda \neq 0,1$, $k \in \Z$ and $n \in \N$, we have
\begin{equation}
\label{eq:FC-AB}
\int_0^1 \lambda^x\B_{n}(x;\lambda)e^{-2\pi ikx}dx=
-\frac{n!}{(2\pi i k - \log \lambda)^{n}}.
\end{equation}
\end{prop}

\begin{proof}
Fixing $k$, we use induction on $n$. One easily checks the case $n=1$, noting that
$\B_1(x;\lambda)=1/(\lambda-1)$. Let $n \geq 2$. Assuming \eqref{eq:FC-AB} is true for $n-1$, integrate it by parts, using the derivative formula \eqref{eq:appell} to
obtain
$$
\int_0^1 \lambda^x\B_{n}(x;\lambda)e^{-2\pi ikx}dx=
\lambda\B_{n}(1;\lambda)-\B_{n}(0;\lambda)-\frac{n!}{(2\pi i k - \log \lambda)^{n}}.
$$
The proof is concluded by checking that $\lambda\B_{n}(1;\lambda) = \B_{n}(0;\lambda)$ for $n \geq 2$. 
This follows from \eqref{eq:GF-AB} noting that $\lambda g(1,\lambda,z) - g(0,\lambda,z) = z$ (see also~\cite[formula 3.5,\,p.~165]{Ap}).
\end{proof}

\begin{remark}
By uniqueness of the Fourier series, \eqref{eq:FC-AB} could be used to \emph{define} the Apostol-Bernoulli polynomials, instead of using the generating function~\eqref{eq:GF-AB}.
Indeed, it is not hard to work ``backwards'' to deduce \eqref{eq:GF-AB} from \eqref{eq:FC-AB}
and also to prove the basic algebraic and differential properties satisfied by the family directly from the Fourier series, thus providing an alternative approach to the theory. This applies also in particular to the Bernoulli and Euler  polynomials.
\end{remark}

For our purposes, it is useful to rewrite \eqref{eq:FS-AB} in the alternative form
\begin{equation}
\label{eq:FS-poles}	
	\frac{1}{n!}\B_{n}(x;\lambda) = - \sum_{a \in S} \frac{e^{a x}}{a^{n}},
\end{equation}
where $S$ is the set of poles of the generating function $g(x,\lambda,z)$, namely
\begin{equation}
\label{def:poles}
	S = \{ a_{k} : k \in Z \},
	\quad
	a_{k} = 2 \pi i k - \log \lambda,
	\quad
	Z = \begin{cases}
	    \Z                 & \text{if $\lambda \neq 1$},
	       \\
	    \Z \setminus \{0\} & \text{if $\lambda = 1$.}
	    \end{cases}
\end{equation}
Note that the derivative relation \eqref{eq:appell} is immediate from \eqref{eq:FS-poles} by differentiating term by term.

This representation also suggests a proof of \eqref{eq:FS-poles} from \eqref{eq:GF-AB} via the calculus of residues. Indeed, $a^{-n} e^{ax} = \res(z^{-(n+1)} g(x,\lambda,z),a)$ for $a \in S$. The left hand side of \eqref{eq:FS-poles} is by definition the coefficient of $z^{n}$ in the power series expansion of $g$ around $0$, so it is equal to $\res(z^{-(n+1)} g(x,\lambda,z),0)$. Hence \eqref{eq:FS-poles} is equivalent to $\sum_{a \in \C} \res(z^{-(n+1)}g(x,\lambda,z),a) = 0$,
which can be proved by showing that the integral $\oint_{C_{N}} z^{-(n+1)}g(x,\lambda,z)\,dz$ tends to $0$ over a suitable increasing sequence of circles~$C_{N}$.

This is essentially the proof in~\cite[Section~1.11]{Bate} of the Fourier expansion of the Lerch transcendent~\eqref{eq:AS-Lerch}. The recent paper~\cite{Bayad} applies this method to obtain the Fourier series of the Apostol-Bernoulli, Apostol-Euler, and also the Apostol-Genocchi polynomials.
Even though it uses no more than basic complex analysis, all things considered, after filling in all the details (for example, the case $n=1$ must be estimated separately from $n \geq 2$), the proof of Proposition~\ref{prop:FC-AB} given above is simpler. Nevertheless, the form \eqref{eq:FS-poles} of the Fourier expansion is particularly well-suited for studying asymptotic approximations to $\B_{n}(z;\lambda)$.

\section{Approximations to $\B_{n}(\lambda)$}
\label{sec:approx-ABN}

To obtain approximation results from \eqref{eq:FS-poles} we need to order the set of poles $S$ of the generating function \eqref{eq:GF-AB} by order of magnitude.

\begin{lemma}
\label{lem:pole-ordering}
Let $a_{k} = 2 \pi i k - \log \lambda$ with $k \in \Z$, $\lambda \in \C$, $\lambda \neq 0$.
\begin{enumerate}
	\item[(a)] If $\Im\lambda > 0$, then for $k \geq 1$, we have
	$$
		0 < |a_{0}| < |a_{1}| < |a_{-1}| < \cdots < |a_{k}| < |a_{-k}| < \cdots
	$$

	\item[(b)] If $\Im\lambda < 0$, then for $k \geq 1$, we have
	$$
		0 < |a_{0}| < |a_{-1}| < |a_{1}| < \cdots < |a_{-k}| < |a_{k}| < \cdots
	$$

	\item[(c)] If $\lambda > 0$, then for $k \geq 1$, we have
	$$
		|a_{0}| < |a_{1}| = |a_{-1}| < \cdots < |a_{k}| = |a_{-k}| < \cdots
	$$
	(note that $a_{0} = 0$ if and only if $\lambda = 1$, in which case it is not a pole, hence is excluded in this chain).

	\item[(d)] If $\lambda < 0$, then for $k \geq 0$, we have
	$$
		0 < |a_{0}| = |a_{1}| < |a_{-1}| = |a_{2}| < \cdots < |a_{-k}| = |a_{k+1}| < \cdots
	$$
\end{enumerate}
In addition, $|a_{k}| \geq 2\pi (|k| - \frac{1}{2})$ if $|k| \geq 1$.
\end{lemma}

\begin{proof}
Let $\xi = \frac{\log\lambda}{2\pi i}$ so that $a_{k} = 2\pi i(k - \xi)$. Since we use the principal branch of the logarithm, this maps $\lambda \in \C \setminus \{0\}$ to the strip
$-\frac{1}{2} < \Re \xi \leq \frac{1}{2}$,
with $\lambda > 0$ corresponding to $\Re\xi = 0$ and $\lambda < 0$ to $\Re\xi = \frac{1}{2}$.
Since $|a_{k}| = 2\pi |k - \xi|$, the above chains are readily verified by considering 
$|x-\xi|^2 = (x-\Re\xi)^2+(\Im\xi)^2$ for real~$x$.
For $k \in \Z$, $|k| \geq 1$, we have $|k - \xi| \geq |k - \Re\xi| \geq |k| - \frac{1}{2}$ and hence $|a_{k}| \geq 2\pi (|k| - \frac{1}{2})$.
\end{proof}

By definition, in an asymptotic expansion, the $n+1$st term should be infinitesimal with respect to the $n$th term. 
Thus we have to consider partial sums of the Fourier series along the chains detailed in Lemma~\ref{lem:pole-ordering} which truncate the chain at a link where there is strict inequality. In other words, we consider only those finite subsets of poles $F \subseteq S$ satisfying 
$$
	\max\{ |a| : a \in F \} < \min\{ |a| : a \in S \setminus F \} \eqdef \mu.
$$
Saying that \eqref{eq:FS-poles} gives an asymptotic expansion means that the
remaining tail is of the order of~$\mu^{-n}$. 

Let us establish some notation. As in \eqref{def:poles}, we let $Z = \Z$ if $\lambda \neq 1$
and $Z = \Z \setminus \{0\}$ if $\lambda = 1$. Letting $a_{k} = 2\pi i k - \log\lambda$, we define
\begin{equation}
\label{def:heads}
	\begin{alignedat}{2}
		    Z_{m} &= \{ k \in Z : |k| \leq m \},
		    	      &\quad
			F_{m} &= \{ a_{k} : k \in Z_{m} \},
		          \\
		Z_{m}^{+} &= \{ k \in Z : |k| \leq m \} \cup \{m+1\},
		    	      &\quad
		F_{m}^{+} &= \{ a_{k} : k \in Z_{m}^{+} \},
		          \\
		Z_{m}^{-} &= \{ k \in Z : |k| \leq m \} \cup \{-(m+1)\},
		    	      &\quad
		F_{m}^{-} &= \{ a_{k} : k \in Z_{m}^{-} \}
	\end{alignedat}
\end{equation}
for an integer $m \geq 0$, noting that for $\lambda = 1$, $Z_{0}$ is empty.
In accordance with Lemma~\ref{lem:pole-ordering}, we have the following possible choices for
$F$, along with the corresponding value(s) of $\mu = \min\{ |a| : a \in S \setminus F \}$:
\begin{equation}
\label{eq:heads-cases}
F = \begin{cases}
	F_{m}, F_{m}^{+},	&
	\\
	F_{m}, F_{m}^{-},	&
	\\
	F_{m},				&
	\\
	F_{m}^{+},			&
	\end{cases}
	\mu = \begin{cases}
		|a_{m+1}|, |a_{-(m+1)}|,	& \text{if $\Im\lambda > 0$,}
		\\
		|a_{-(m+1)}|, |a_{m+1}|,	& \text{if $\Im\lambda < 0$,}
		\\
		|a_{m+1}|,				& \text{if $\lambda > 0$ ($m > 0$ if $\lambda = 1$),}
		\\
		|a_{-(m+1)}|,			& \text{if $\lambda < 0$.}
	\end{cases}
\end{equation}
Note that $\mu = |a_{\pm(m+1)}| \geq 2\pi( m + \frac{1}{2}) \geq \pi$ in each of these cases.

We begin our study of \eqref{eq:FS-poles} by proving that it gives an asymptotic expansion for the Apostol-Bernoulli numbers $\B_{n}(0;\lambda) = \B_{n}(\lambda)$. 

\begin{prop}
\label{prop:approx-ABN}
Given $\lambda \in \C$, $\lambda \neq 0$, let $F$ be a finite subset of the set of poles $S$ of the generating function \eqref{eq:GF-AB} of $\B_{n}(x;\lambda)$ satisfying
$$
	\max\{ |a| : a \in F \} < \min\{ |a| : a \in S \setminus F \} = \mu.
$$
For all integers $n \geq 2$, we have
\begin{equation}
\label{eq:approx-ABN}
	\frac{\B_{n}(\lambda)}{n!} = -\sum_{a \in F} \frac{1}{a^{n}} + O(\mu^{-n}),
\end{equation}
where the constant implicit in the order term depends only on $\lambda$ and $F$.
\end{prop}

In this sense then, using the appropriate approximating sums over the sets $F$, the Fourier series \eqref{eq:FS-poles} of $\B_{n}(x;\lambda)$ at $x=0$ is an asymptotic expansion for the Apostol-Bernoulli numbers as $n \to \infty$.

\begin{proof}
Relabel the set of poles in increasing order of magnitude as
$|\alpha_{0}| \leq |\alpha_{1}| \leq \cdots \leq |\alpha_{M}| \leq \cdots$.
The estimate $|a_{k}| \geq 2\pi(|k|-\frac{1}{2})$ from Lemma~\ref{lem:pole-ordering}
shows that $\sum_{k} \alpha_{k}^{-n}$ is absolutely convergent for $n \geq 2$. For any $M \geq 0$, we have
$$
\begin{aligned}
	   	  \sum_{k = M+1}^{\infty} \frac{1}{|\alpha_{k}|^{n}}
	&=    \frac{1}{|\alpha_{M+1}|^{n}} 
	      \sum_{k = M+1}^{\infty} \left| \frac{\alpha_{M+1}}{\alpha_{k}}\right|^{n}
	\\
	&\leq \frac{1}{|\alpha_{M+1}|^{n}} 
	      \sum_{k = M+1}^{\infty} \left| \frac{\alpha_{M+1}}{\alpha_{k}}\right|^{2}
	=     \frac{c_{M,\lambda}}{|\alpha_{M+1}|^{n}}, 
\end{aligned}$$
where $c_{M,\lambda}$ is a constant depending only on $M$ and $\lambda$. This applies to any tail, in particular to the tail $S \setminus F$.
\end{proof}

\begin{remark}
\label{rk:O-constants}
We can say more about the constant in the above proof by considering the Hurwitz zeta function for real values of the parameters. For example, for $F = F_{m} = \{ k \in \Z : |k| \leq m \}$ (excluding $0$ if $\lambda = 1$) and $\Im\lambda > 0$, for which the next pole in order of magnitude is $|a_{m+1}|$, if we denote the constant by $c_{m,\lambda}$ and let
$\xi = \frac{\log\lambda}{2\pi i}$, we have, by Lemma~\ref{lem:pole-ordering},
$$
\begin{aligned}
            c_{m,\lambda}
       &=   |\alpha_{m+1}|^{2} \sum_{|k| \geq m+1}^{\infty} \frac{1}{|a_{k}|^{2}}
	   =    |m+1 - \xi|^{2}
	        \sum_{|k| \geq m+1}^{\infty} \frac{1}{|k - \xi|^{2}}
	  \\
	  &\leq |m+1 - \xi|^{2}
	        \sum_{|k| \geq m+1}^{\infty} \frac{1}{\left(|k| - \frac{1}{2}\right)^{2}}
	   =    2 |m+1 - \xi|^{2} \zeta(2,m + \tfrac{1}{2}).
\end{aligned}
$$
Simply comparing the sum $\zeta(\sigma,q) = \sum_{k=0}^{\infty} (k + q)^{-\sigma}$
for $\sigma > 1$, $q > 0$ with the corresponding integral yields
$$
	\zeta(\sigma,q) < \left(1 + \frac{q}{\sigma - 1}\right) q^{-\sigma},
$$
which applied to the above estimate for $c_{m,\lambda}$ gives
$$
	     c_{m,\lambda}
	\leq 2 |m+1-\xi|^{2} \left(1 + \frac{m + \tfrac{1}{2}}{n-1}\right)
	     \frac{1}{\left(m + \tfrac{1}{2}\right)^{2}}.
$$
Since $\lim_{m \to \infty} (m+1-\xi)/(m+\frac{1}{2}) = 1$, this shows that for fixed $m$ and $n \gg 0$, $c_{m,\lambda}$ is bounded independent of $m$ and we can then replace it with a constant $c_{\lambda}$ depending only on $\lambda$. Thus the tail can be estimated by $c_{\lambda}|a_{m+1}|^{-n}$ for fixed $m$ and $n \gg 0$.
In general, for any $n \geq 2$, we still have an estimate of the form
$c_{\lambda} |m+1-\xi|$ and hence the tail can always be estimated by $c_{\lambda}|a_{m+1}|^{-(n-1)}$.
\end{remark}

\section{Approximations to $\B_{n}(z;\lambda)$ on the complex plane}
\label{sec:approx-AB}

We now come to the central result of this paper, that the Fourier series \eqref{eq:FS-poles} of $\B_{n}(x;\lambda)$ for $x \in [0,1]$ extends to the complex plane as an asymptotic expansion 
for $\B_{n}(z;\lambda)$ given $z \in \C$. We state this in a more precise form, since we also obtain information on the implicit constants in the asymptotic approximation.

\begin{theorem}
\label{thm:approx-AB}
Given $\lambda \in \C$, $\lambda \neq 0$, let $F$ be a finite subset of the set of poles $S$ of the generating function \eqref{eq:GF-AB} of $\B_{n}(x;\lambda)$ satisfying
$$
	\max\{ |a| : a \in F \} < \min\{ |a| : a \in S \setminus F \} \eqdef \mu.
$$
For all integers $n \geq 2$, we have, uniformly for $z$ in a compact subset $K$ of~$\C$,
\begin{equation}
\label{eq:approx-AB}
	\frac{\B_{n}(z;\lambda)}{n!} = -\sum_{a \in F} \frac{e^{az}}{a^{n}} 
	                               + O\left(\frac{e^{\mu|z|}}{\mu^{n}}\right),
\end{equation}
where the constant implicit in the order term depends on $\lambda, F$ and $K$.

In fact, for $n \gg 0$, the order constant may be taken equal to the value of the constant for the Apostol-Bernoulli numbers, corresponding to $z = 0$, thus eliminating its dependence on $K$ (how large $n$ has to be of course still depends on the compact set $K$).
\end{theorem}

\begin{proof}
We will use Proposition~\ref{prop:approx-ABN}, which is the case $z=0$, and the ``binomial formula''
\begin{equation}
\label{eq:binom}
	\B_{n}(x+y;\lambda) = \sum_{k=0}^{n} \binom{n}{k} \B_{n-k}(x;\lambda) \, y^{k}
\end{equation}
which may be proved directly from the generating function \eqref{eq:GF-AB} or from the derivative relation~\eqref{eq:appell}. It is well-known that for a given polynomial family, in fact \eqref{eq:binom} is equivalent to~\eqref{eq:appell}
(for fixed~$x$, \eqref{eq:binom} is the Taylor expansion of $\B_{n}(x+y;\lambda)$ in powers of~$y$).
Both serve to define the notion of an Appell sequence. For $z \in \C$, writing $z = z + 0$, \eqref{eq:binom} and Proposition~\ref{prop:approx-ABN} yield
$$
\begin{aligned}
	   \frac{\B_{n}(z;\lambda)}{n!}
	&= \sum_{k=0}^{n} \frac{\B_{n-k}(\lambda)}{(n-k)!} \, \frac{z^{k}}{k!}
	 = \sum_{k=0}^{n} \left(- \sum_{a \in F} \frac{1}{a^{n-k}} + O(\mu^{-n+k})\right) 
	   \, \frac{z^{k}}{k!}
	\\
	&= - \sum_{a \in F} \sum_{k=0}^{n} \frac{1}{a^{n-k}} \frac{z^{k}}{k!}
	   + \sum_{k=0}^{n} O(\mu^{-n+k}) \frac{z^{k}}{k!},
\end{aligned}
$$
where the implicit constant $c$ is that corresponding to $z = 0$ and only depends on $F,\lambda$.
Now, consider the partial sums and tails of the exponential series,
$$
	e_{n}(w) = \sum_{k=0}^{n} \frac{w^{k}}{k!},
	\quad
	e_{n}^{*}(w) = e^{w} - e_{n}(w) = \sum_{k=n+1}^{\infty} \frac{w^{k}}{k!}.
$$
The first summand above is
$$
	  - \sum_{a \in F} \frac{e_{n}(az)}{a^{n}}
	= - \sum_{a \in F} \frac{e^{az}}{a^{n}} + \sum_{a \in F} \frac{e_{n}^{*}(az)}{a^{n}}.
$$
To prove the theorem, we need to show that the resulting extra terms have the correct order of magnitude. Indeed, 
$$
	     \left|\sum_{k=0}^{n} O(\mu^{-n+k}) \frac{z^{k}}{k!}\right|
	\leq c \sum_{k=0}^{n} \mu^{-n+k} \, \frac{|z|^{k}}{k!}
	=    c \mu^{-n} \sum_{k=0}^{n} e_{n}(\mu|z|)
	\leq c \mu^{-n} e^{\mu|z|}.
$$
The tails of the exponential series may be estimated using the complex version of the Lagrange remainder for Taylor series (which is actually an upper bound for the modulus rather than an equality as in the real case):
$$
	|e_{n}^{*}(w)| \leq e^{\Re^{+}(w)} \frac{|w|^{n+1}}{(n+1)!},
	\quad
	\Re^{+}(w) = \max\{\Re(w), 0\}.
$$
Since $|a| < \mu$ for all $a \in F$, we have
$$
	     \left|\frac{e_{n}^{*}(az)}{a^{n}}\right|
	\leq |a| e^{|az|}\frac{|z|^{n+1}}{(n+1)!}  
	<    \mu e^{\mu |z|} \frac{|z|^{n+1}}{(n+1)!}  
$$
so that 
$$
	     \left|\sum_{a \in F} \frac{e_{n}^{*}(az)}{a^{n}}\right|
	\leq \#F \mu e^{\mu |z|} \frac{|z|^{n+1}}{(n+1)!}
$$
(here $\#F$ denotes the number of elements in~$F$)
and the latter term is less than $c e^{\mu |z|} \mu^{-n}$ when
$$
	\#F \frac{(\mu |z|)^{n+1}}{(n+1)!} < c,
$$
a condition which certainly holds for $n \gg 0$, uniformly for $z$ in a compact subset 
$K \subseteq \C$. In general, since this sequence is bounded, for any $n \geq 2$ we obtain a constant independent of~$n$, depending on $\lambda,F,K$.
\end{proof}

\begin{remark}
Since the point is that we get an asymptotic series in $n$, the $e^{\mu|z|}$ term could be absorbed into the constant in the order term, but doing this hides the numerical behavior of the approximation. Indeed, it is quickly apparent in computation that, even for rather small values of $|z|$, the rapid growth of this factor offsets the action of $\mu^{-n}$ until $n$ is quite large.
\end{remark}

\begin{example}[Bernoulli polynomials] As a special case of the theorem, we derive Dilcher's results in~\cite{Dil} for the approximation of Bernoulli polynomials.
\end{example}

\begin{corollary}
\label{cor:approx-B}
The Bernoulli polynomials satisfy, uniformly on a compact subset $K$ of~$\C$, the estimates
\begin{equation}
\label{eq:approx-B}
\begin{aligned}
     \frac{(-1)^{n-1} (2\pi)^{2n}}{2 (2n)!} \,B_{2n}(z)
  &= \cos(2\pi z) + O\left(\frac{e^{4\pi|z|}}{2^{n}}\right),
  \\[2pt]
     \frac{(-1)^{n-1} (2\pi)^{2n+1}}{2 (2n+1)!} \,B_{2n+1}(z)
  &= \sin(2\pi z) + O\left(\frac{e^{4\pi|z|}}{2^{n}}\right),
\end{aligned}
\end{equation}
where the implicit constant depends on $K$. Moreover, for $n \gg 0$ this constant can be made
independent of $K$, equal to the constant for the Bernoulli numbers, corresponding
to the case $z = 0$.
\end{corollary}

\begin{proof}
The Bernoulli case corresponds to $\lambda = 1$, where $\xi = \frac{\log\lambda}{2\pi i} = 0$ and the set of poles is $S = \{ 2 \pi i k : k \in \Z, k \neq 0 \}$. Hence Theorem~\ref{thm:approx-AB}, after multiplying by $(2\pi i)^{n}$, states that
\begin{equation}
\label{eq:approx-Ber}
	  \frac{(2\pi i)^{n} B_{n}(z)}{n!}
	= - \sum_{0 < |k| \leq m} \frac{e^{2\pi i k z}}{k^{n}}
	  + O\left(\frac{e^{2\pi(m+1)|z|}}{(m+1)^{n}}\right).
\end{equation}
The result stated above is the case $m=1$.
\end{proof}

\section{Asymptotic behavior of $\B_{n}(z;\lambda)$}
\label{sec:asympt}

From now on, assume $\lambda \neq 1$. Then the pole set $S = \{ a_{k} = 2\pi i k - \log\lambda : k \in \Z \}$ contains $a_{0} = -\log\lambda$ and hence the series \eqref{eq:approx-AB} 
begins with the term $-e^{a_{0}z} a_{0}^{-n} = (-1)^{n-1} \lambda^{-z} \log^{n}\lambda$. It makes sense to normalize the series by considering the modified expression
\begin{equation}
\label{def:ab-normalized}
	\beta_{n}(z;\lambda) = (-1)^{n-1} \frac{\log^{n} \lambda}{n!} \lambda^{z} \B_{n}(z;\lambda).
\end{equation}
Now~\eqref{eq:FS-AB} becomes
\begin{equation}
\label{eq:FS-AB-mod}
	  \beta_{n}(x;\lambda)
	= \sum_{k \in \Z} \frac{e^{2 \pi i k x}}{(1 - \Lambda k)^{n}},
	\quad
	x \in [0,1],\
	\Lambda = \frac{2\pi i}{\log\lambda}.
\end{equation}
It is straightforward to check that the transformation
$\Lambda = \frac{2\pi i}{\log \lambda}$ maps the region 
$\C \setminus ((-\infty,0] \cup \{1\})$ to the region $\{ \Lambda \in \C : |\Lambda-1| > 1, |\Lambda+1| > 1 \}$ which is the exterior of the ``figure eight'' formed by the union of the two closed disks $|\Lambda \pm 1| \leq 1$, tangent at~$0$. The negative real axis $(-\infty,0)$ is mapped to the circumference $|\Lambda-1|=1$ minus $\Lambda=0$. 
Since $|1-k\Lambda|=|1-k\xi^{-1}| = |\xi^{-1}| |k -\xi|$, where $\xi = \frac{\log\lambda}{2\pi i}$, the ordering of the terms is the same as in Lemma~\ref{lem:pole-ordering}.

The asymptotic approximation \eqref{eq:approx-AB} changes accordingly.

\begin{prop}
\label{prop:approx-AB-mod}
Let $\lambda \in \C$, $\lambda \neq 0,1$.
For any integer $m \geq 0$ we have, for $z$ in a compact subset $K$ of $\C$,
\begin{equation}
\label{eq:approx-AB-mod}
	(-1)^{n-1} \frac{\log^{n} \lambda}{n!} \lambda^{z} \B_{n}(z;\lambda)
	= \sum_{k \in I_{m}} \frac{e^{2\pi i k z}}{\bigl(1- \frac{2\pi i k}{\log\lambda}\bigr)^{n}}
	                       + O\Biggl(
	                          \frac{e^{(2|\log\lambda|+2\pi(m+1))|z|}}
	                               {\Bigl|1 \pm \frac{2\pi i(m+1)}{\log\lambda}\Bigr|^{n}}
	                          \Biggr).
\end{equation}
The constant implicit in the order term depends on $\lambda,m$ and $K$. For $n \gg 0$ it can be taken equal to the constant for the case $z=0$, thus making it independent of $K$.
Here $I_{m} = Z_{m}$ or $Z_{m}^{\pm}$ are as in \eqref{def:heads},
with the signs chosen according to~\eqref{eq:heads-cases}.
\end{prop}

\begin{proof}
It is straightforward to check that the general term in the sum changes to the expression above, and the order term gets multiplied by $|\lambda^{z} \log^{n}\lambda|$. Estimating
$|\lambda^{z}| \leq e^{|z\log\lambda|}$ and $\mu = |a_{\pm(m+1)}| \leq 2\pi(m+1) + |\log\lambda|$ gives the order term.
\end{proof}

\begin{remark}
If we don't incorporate the term $\lambda^{z}$ into the definition \eqref{def:ab-normalized} of 
$\beta_{n}$ above, we can drop the ``$2$'' from $|\log\lambda|$ inside the $O$ term.
\end{remark}

\begin{corollary}
\label{cor:0th-approx}
For $\lambda \in \C$, $\lambda \neq 0,1$ and $\lambda \notin (-\infty,0)$, we have, for $z$ in a compact subset $K$ of $\C$,
$$
	(-1)^{n-1} \frac{\log^{n} \lambda}{n!} \lambda^{z} \B_{n}(z;\lambda)
	= 1 
      + O\left(
         \frac{e^{(2|\log\lambda|+2\pi)|z|}}{\min\left|1 \pm \frac{2\pi i}{\log\lambda}\right|^{n}}
         \right).
$$
The constant implicit in the order term depends on $\lambda$ and $K$. However, for $n \gg 0$ it can be taken equal to the constant when $z=0$.
In particular
$$
	\lim_{n \to \infty} (-1)^{n-1} \frac{\log^{n} \lambda}{n!} \B_{n}(z;\lambda)
	= \lambda^{-z}
$$
uniformly on compact subsets of~$\C$.
\end{corollary}

\begin{proof}
This is the case $m=0$ above, with the choice of $I_{0} = Z_{0}=\{0\}$.
\end{proof}

\section{Oscillatory phenomena}
\label{sec:oscillatory}

When $\lambda \in \R$, $\lambda \neq 0$, Lemma~\ref{lem:pole-ordering} shows that there are poles of equal modulus, which must be grouped in pairs in the asymptotic approximation. These pairs are indexed differently depending on the sign of $\lambda$, but their behavior turns out to be the same.
For $\lambda > 0$, equal moduli poles correspond to pairs of integers $\pm k$ with $k \geq 1$, and are conjugate:
$$
	a_{k} = 2\pi i k - \log\lambda, \quad a_{-k} = -2\pi i k - \log\lambda = \overline{a_{k}}.
$$
Writing the poles in polar form, $a_{k} = \rho_{k} e^{2\pi i \alpha_{k}}$ with $\rho_{k} = |a_{k}| > 0$ and $\alpha_{k} \in [0,1]$, this pair contributes to the asymptotic expansion \eqref{eq:approx-AB} with
$$
\begin{aligned}
	   \frac{e^{a_{k}z}}{a_{k}^{n}} + \frac{e^{\overline{a_{k}}z}}{\overline{a_{k}}^{\, n}}
	&= \lambda^{-z }\rho_{k}^{-n} 
	   \left(
			e^{2\pi i k z} e^{-2\pi i n \alpha_{k}}
		  + e^{-2\pi i k z} e^{2\pi i n \alpha_{k}}
	   \right).
\end{aligned}
$$
This simplifies to
\begin{equation}
\label{eq:cos-pos}
	2 \lambda^{-z} \rho_{k}^{-n} \cos(2 \pi (kz - n \alpha_{k})).
\end{equation}

Similarly, for $\lambda < 0$, equal moduli poles correspond to pairs of integers $\{-k,k+1\}$ with $k \geq 0$. It is easy to see that these are also conjugate:
$$
\begin{aligned}
	 a_{-k} &= -2\pi i k - \log\lambda
	         = -(2k+1)\pi i - \log(-\lambda),
	        \\
	a_{k+1} &= 2\pi i(k+1) - \log\lambda
	         = (2k+1)\pi i - \log(-\lambda)
\end{aligned}
$$
and hence, now writing $a_{k+1} = \rho_{k} e^{2\pi i \alpha_{k}}$ with $\rho_{k} = |a_{k+1}| > 0$ and $\alpha_{k} \in [0,1]$, the same reasoning shows that this pair contributes to the asymptotic expansion with 
\begin{equation}
\label{eq:cos-neg}
	2 (-\lambda)^{-z} \rho_{k}^{-n} \cos(\pi ((2k+1) z - 2 n \alpha_{k})).
\end{equation}

These terms cannot vanish unless $z \in \R$, in which case we will write $z = x$. When this is the case, having fixed $k$ and $x$, the way that the expressions in \eqref{eq:cos-pos} and \eqref{eq:cos-neg} vary with $n$ depends on the behavior of the sequences $kx - n \alpha_{k}$ and $(2k+1)x - 2n\alpha_{k}$ modulo~$1$, or equivalently, their fractional parts. 

The study of the fractional parts of $\beta - n \alpha$ where $\alpha,\beta \in \R$ are fixed and $n$ varies is the problem of real inhomogeneous Diophantine approximation. This is a difficult problem with an extensive literature and important ramifications. Its origins lie in Kronecker's Theorem, which states that if $\alpha$ is irrational, the fractional parts are dense in $[0,1]$. In fact, they are equidistributed modulo~$1$.
Of course if $\alpha$ is rational, the sequence is periodic of period the denominator of~$\alpha$. 


As far as we are concerned, then, such oscillatory terms may be considered ``chaotic'' unless
$\alpha_{k}$ is rational, which corresponds to the quotient of the conjugate pair of poles 
$a_{k}, \overline{a_{k}}$ being a root of unity.

The simplest appearance of an oscillatory term occurs in the ``0th order'' case of \eqref{eq:approx-AB} for 
$\lambda \in (-\infty,0)$.

\begin{prop}
\label{prop:l<0}
Let $\lambda < 0$. Then as $n \to \infty$, for $z$ in a compact subset $K$ of $\C$,
$$
	  \frac{(-1)^{n-1}\log^{n}\lambda}{n!} \lambda^{z} \B_{n}(z;\lambda)
	= 1 + \left(\frac{\log|\lambda|+\pi i}{\log|\lambda|-\pi i}\right)^{n} e^{2\pi i z}
	  + O\left(
	     \frac{e^{2(\pi + |\log\lambda|)|z|}}{\left|1 + \frac{2\pi i}{\log\lambda}\right|^{n}}
	     \right).
$$
The constant implicit in the order term depends only on $\lambda$ and $K$. Moreover, for $n \gg 0$ it may be taken equal to the corresponding constant for $z=0$, thus making it independent of $K$.

The part of the approximating terms in parentheses has modulus $1$, hence is periodic or dense in the unit circle according as its argument is a rational or irrational multiple of~$2\pi$.
In particular,
$$
	  \lim_{n \to \infty}
	  \left( \frac{(-1)^{n-1}\log^{n}\lambda}{n!} \lambda^{z} \B_{n}(z;\lambda)
	- \left(\frac{\log|\lambda|+\pi i}{\log|\lambda|-\pi i}\right)^{n} e^{2\pi i z} \right)
	= 1
$$
uniformly on compact subsets of~$\C$.
\end{prop}

\begin{example}
\label{ex:l=-1}
It is easy to check that when $\lambda < 0$, periodic behavior occurs in the $0$th order approximation if and only if $\lambda = -e^{\pi \cot \frac{\pi k}{d}}$ for integers
$d \geq 1$ and $k$.
A special case is $\lambda = -1$, where we obtain, after simplifying,
\begin{equation}
\label{eq:l=-1}
\begin{aligned}
	   \frac{(-1)^{n-1}\pi^{2n}}{2(2n)!} \B_{2n}(z;-1) 
	&= \cos\pi z + O\left(\frac{e^{3\pi |z|}}{3^{2n}}\right),
	\\
	   \frac{(-1)^{n-1}\pi^{2n+1}}{2(2n+1)!} \B_{2n+1}(z;-1) 
	&= \sin\pi z + O\left(\frac{e^{3\pi |z|}}{3^{2n+1}}\right).
\end{aligned}\end{equation}
These approximations are essentially those of the Euler polynomials (see Section~\ref{sec:ApostolEuler}).
\end{example}

\begin{remark}
\label{eq:succesive-periods}
An interesting question related to oscillatory behavior is whether we can have repeated instances of ``good'', i.e., periodic, behavior. In other words, can we have different pairs of conjugate poles whose quotients are roots of unity? When $\lambda > 0$, considering the M\"obius transformation $M(z) = \frac{1 + iz}{1 - iz}$, which satisfies
$M(\tan z) = e^{2iz}$, we have, for $k \geq 1$,
$$
	  \frac{a_{k}}{a_{-k}}
	= \frac{1 - i \frac{2\pi k}{\log\lambda}}{1 + i \frac{2\pi k}{\log\lambda}}
	= M \bigl(-\tfrac{2\pi k}{\log\lambda}\bigr)
	= e^{-2\pi i \theta_{k}}
	\iff
	\frac{2\pi k}{\log\lambda} = \tan \pi \theta_{k},
$$
and hence the quotient is a root of unity if and only if $\frac{2\pi k}{\log\lambda} = \tan\pi \theta_{k}$ with $\theta_{k}  \in \Q$ satisfying $0 < |\theta_{k}| < \frac{1}{2}$ (this expression cannot be zero). The situation for $\lambda < 0$ is similar, resulting in
$$
	\frac{(2k + 1)\pi}{\log(-\lambda)} = \tan \pi \theta_{k},
	\quad
	\theta_{k} \in \Q,\
	0 < |\theta_{k}| < \tfrac{1}{2}.
$$
If this happens for two different pairs, corresponding to $k,l \geq 1$, say, then
$$
	\frac{\tan \pi \theta_{k}}{\tan \pi \theta_{l}} = \frac{k}{l}, \frac{2k+1}{2l+1},
	\quad
	\theta_{k},\theta_{l} \in \Q, \
	0 < |\theta_{k}|,|\theta_{l}| < \tfrac{1}{2},
$$
respectively according as $\lambda > 0$ or $\lambda < 0$. 
Without loss of generality one may assume both $\theta_{k},\theta_{l} > 0$.
For example, $\tan \frac{\pi}{3} / \tan \frac{\pi}{6} = 3$ so there is a solution pair
$(k,3k)$ for $\lambda = e^{2 \pi k \sqrt{3}}$.
However, the ratio $k/l = 2$ is impossible, as can be seen by considering cyclotomic polynomials. What happens in general? 

A similar phenomenon was studied by the authors for Bernoulli polynomials in~\cite{NRV-JAT}, namely, when the first conjugate term in the Fourier series fails to provide information because it vanishes, one turns to the next, which in that case cannot vanish simultaneously with the first.
\end{remark}

\section{Successive quotients}

\begin{prop}
\label{prop:quot-AB}
Let $\lambda \notin (0,+\infty)$ and $\beta_{n}(z;\lambda)$ as in~\eqref{def:ab-normalized}. 
Then $\beta_{n}(z;\lambda) \neq 1$ for $n \gg 0$ and 
if we denote by $\epsilon = \pm 1$ the sign of $\Im \lambda$ when $\lambda \notin \R$, and
$\epsilon = 1$ when $\lambda \in (-\infty,0)$, then we have
$$
	  \frac{\beta_{n+1}(z;\lambda) - 1}{\beta_{n}(z;\lambda) - 1}
	= \frac{1}{1 - \epsilon\frac{2\pi i}{\log\lambda}} 
	  + O\left(
	     \left|
	             \frac{1 - \epsilon\tfrac{2\pi i}{\log\lambda}}
	                  {1 + \epsilon\tfrac{2\pi i}{\log\lambda}}
	     \right|^{n}
	     \right),
	     \quad
	     n \to \infty, 
$$
uniformly for $z$ in a compact subset of~$\C$. In particular
$$
	  \lim_{n} \frac{\beta_{n+1}(z;\lambda) - 1}{\beta_{n}(z;\lambda) - 1}
	= \frac{1}{1 - \epsilon\frac{2\pi i}{\log\lambda}}
$$
uniformly on compact subsets of~$\C$.
\end{prop}

\begin{proof}
If $\lambda \notin \R$ we may assume without loss of generality that $\Im \lambda > 0$, so that $\epsilon = 1$. 
Applying Proposition~\ref{prop:approx-AB-mod} with index set $I = Z_{0}^{+} = \{0,1\}$, we have
$$
	\beta_{n}(z;\lambda) = 1 + \frac{e^{2\pi i z}}{(1-\Lambda)^{n}} + O(|1+\Lambda|^{-n})
$$
where $\Lambda = \frac{2\pi i}{\log\lambda}$ and for simplicity we have subsumed the exponential in $|z|$ into the 
order term. Therefore
$$
	  \beta_{n}(z;\lambda) - 1 
	= \frac{1}{(1-\Lambda)^{n}} \left(e^{2\pi i z} + O\left(\left|\frac{1-\Lambda}{1+\Lambda}\right|^{n}\right)\right).
$$
Since $|1-\Lambda| < |1+\Lambda|$ the $e^{2\pi i z}$ term dominates for $n \gg 0$, hence
$\beta_{n}(z;\lambda) - 1 \neq 0$ for $n \gg 0$. Division of the approximations for $n$ and 
$n+1$ then yields the result.

By Lemma~\ref{lem:pole-ordering} and \eqref{eq:heads-cases}, for $\lambda \in (-\infty,0)$ we have the same situation, since $I = Z_{0}^{+}$ is also a suitable index set. In this case, we have $|1-\Lambda| = 1 < |1 + \Lambda|$.
\end{proof}

\begin{remark}
Compare this result with Corollary~\ref{cor:0th-approx}, which shows that $\beta_{n}(z;\lambda) - 1$ tends to $0$
uniformly on compact subsets of $\C$ for $\lambda \notin (-\infty,0)$. As we saw in Proposition~\ref{prop:l<0},
for $\lambda < 0$, the difference $\beta_{n}(z;\lambda) - 1$ oscillates, yet by
Proposition~\ref{prop:quot-AB}, its successive quotients approach a limit.
\end{remark}

When $\lambda \in \R^{+}$, we cannot use $Z_{0}^{+}$ but rather $Z_{1} = \{0,1,-1\}$ as index set. This brings into play
the oscillatory phenomena mentioned in Section~\ref{sec:oscillatory}, which prevent us from obtaining an analogous clean result. However, we can still deduce some bounds from \eqref{eq:cos-pos}, or rather a slight modification after normalizing.

\begin{prop}
\label{prop:quot-AB-R+}
Let $\beta_{n}$ be as in \eqref{def:ab-normalized}.
Let $\lambda \in (0,\infty)$, $\lambda \neq 1$, $\Lambda = 2\pi i/\log\lambda$ and
$\rho = |1 - \Lambda| = |1 + \Lambda|$, $\mu = |1 - 2 \Lambda|$, so that $1 < \rho < \mu$. 
Let $\eta = \rho \mu^{-1}$, so $0 < \eta < 1$.
Then there is a constant $c > 0$ independent of $n$ such that
\begin{equation}
\label{eq:quot-AB-R+}
         \rho^{-1} |\tanh (2\pi \Im z)| - c \eta^{n}
	\leq \left| \frac{\beta_{n+1}(z;\lambda) - 1}{\beta_{n}(z;\lambda) - 1}\right|
	\leq \rho^{-1} |\coth (2\pi \Im z)| + c \eta^{n}.
\end{equation}
The estimate holds uniformly in $n$ for $z$ in compact subsets of $\C \setminus \R$.
\end{prop}

\begin{proof}
Let $|1-\Lambda| = |1+\Lambda| = \rho$, and $\mu = |1 - 2\Lambda|$. By Lemma~\ref{lem:pole-ordering},
we have $1 < \rho < \mu$. Now write $1-\Lambda = \rho e^{2\pi i \alpha}$ with $\alpha \in [0,1]$. Then as in 
\eqref{eq:cos-pos}, we have
$$
\begin{aligned}
   	\beta_{n}(z;\lambda) - 1 &= 2 \rho^{-n}\cos 2\pi(z - n\alpha) + O(\mu^{-n})
							\\
	                         &= 2 \rho^{-n} \bigl( \cos 2\pi(z-n\alpha) + O(\eta^{n}) \bigr).
\end{aligned}
$$
In general, if $z = x + i y$ with $x,y \in \R$, then $|\cos z|^{2} = \cos^{2} x + \sinh^{2} y$. 
In our case,
$
	|\cos 2\pi(z - n\alpha)|^{2} = \cos^{2} 2\pi(x - n\alpha) + \sinh^{2} 2\pi y.
$
From the remarks made in Section~\ref{sec:oscillatory}, unless $\alpha$ is rational, we have no control over the term 
$\cos 2\pi(x-n\alpha)$, which will be dense in $[-1,1]$. However, if $z \notin \R$ then at least we have a nontrivial lower bound
$
	|\cos 2\pi(z - n\alpha)| \geq |\sinh 2\pi y| > 0,
$
which eventually dominates the $O((\rho\mu^{-1})^{n})$ term.
Combining this with the upper bound $|\cos 2\pi(z - n\alpha)| \leq \cosh 2\pi y$ and forming the successive quotients gives the result.
\end{proof}

The bounds \eqref{eq:quot-AB-R+} are sharp. They correspond to the cases when
$x-n\alpha \equiv 0,\frac{1}{4},\frac{3}{4}$ modulo~$1$, where $\alpha$ is defined
by $\omega = \frac{1-\Lambda}{|1-\Lambda|} = e^{2\pi i \alpha}$. When $\alpha \notin \Q$
the sequence $x - n\alpha$ will come arbitrarily close modulo~$1$ to these 
values infinitely often.

If we have $z = x \in \R$ then the cosine term in the approximation
$\beta_{n}(x;\lambda) - 1 = 2\rho^{-n}( \cos 2\pi(x - n\alpha) + O(\eta^{n}))$ will oscillate densely in $[-1,1]$ unless $\alpha$ is rational. In particular it does not dominate the order term and nothing guarantees that 
$\beta_{n}(x;\lambda) - 1$ is nonzero. Numerical examples where it does vanish are easy to find.

This leaves us with the case $z = x \in \R$ where $\alpha$ is rational. Equivalently, $\omega = \frac{1-\Lambda}{|1-\Lambda|} = e^{2\pi i \alpha}$ is a root of unity. Here one can prove the following result.

\begin{prop}
\label{prop:quot-AB-R+Q}
Let $\beta_{n}$ be as in \eqref{def:ab-normalized}.
Let $\lambda \in (0,\infty)$, $\lambda \neq 1$, $\Lambda = 2\pi i/\log\lambda$,
$\rho = |1 - \Lambda| = |1 + \Lambda|$, and $\mu = |1 - 2\Lambda|$, so that $1 < \rho < \mu$. 
Set $\eta = \rho \mu^{-1}$, so $0 < \eta < 1$. Suppose that $\omega = \rho^{-1} (1-\Lambda) = e^{2\pi i \alpha}$ is a root of unity, i.e.\ $\alpha$ is rational. Then for $z$ in a compact subset $K$ of $\C$ at positive distance $\delta$ from the exceptional set $E(\alpha) \subseteq \R$ where  $\cos(x - n\alpha) = 0$ for some $n \in \N$, there is a constant $c = c_{K,\lambda}> 0$ such that
\begin{equation}
\label{eq:prop:quot-AB-R+Q}
         \rho^{-1}
         \left( \frac{4\delta}{\cosh 2 \pi y} - c \eta^{n} \right)
	\leq \left| \frac{\beta_{n+1}(z;\lambda)-1}{\beta_{n}(z;\lambda)-1}\right|
	\leq \rho^{-1}
	     \left( \frac{\cosh 2\pi y}{4\delta} + c \eta^{n} \right).
\end{equation}
\end{prop}

\begin{proof}
Suppose $\alpha = \frac{a}{d}$ with $a,d \in \Z, d > 0$ and $\gcd(a,d)=1$. Then $\cos 2\pi(x - n\alpha)$ is periodic in $n$ of period $d$.
It can be checked that $E(\alpha) = \frac{1}{4d} + \frac{1}{2d}\Z$, $\frac{1}{2d} + \frac{1}{d} \Z$ or $\frac{1}{d} \Z$ respectively, according to whether $\gcd(4a,d) = 1,2$ or~$4$. In any case $E(\alpha) \subseteq \frac{1}{4d} \Z$
and in fact these three affine lattices are disjoint with union $\frac{1}{4d}\Z$.

As far as estimation is concerned, we may replace the cosine with the function 
$\|\xi\|$ denoting the distance from $\xi$ to the nearest integer; namely, we have
$$
		     2 \left\|2x -\tfrac{1}{2} \right\|
		\leq |\cos 2\pi x|
		\leq \pi \left\|2x -\tfrac{1}{2} \right\|.
$$
In our case, it is easily seen that
$
	  \min_{n \in \Z} \left\| 2x - \tfrac{1}{2} - 2n \alpha \right\|
	= 2\dist(x,E(\alpha)).
$
The rest of the proof is the same as in Proposition~\ref{prop:quot-AB-R+}, using these estimates to bound
$\cos 2\pi(z - n\alpha)$ below in terms of the distance to the exceptional set, rather than using the bound 
$|\sinh 2\pi y|$ as before.
\end{proof}

\section{The Apostol-Euler Polynomials}
\label{sec:ApostolEuler}

The Apostol-Euler polynomials are a generalization of the Euler polynomials, introduced by Luo and Srivastava (see~\cite{Luo-Tai,LuoSri}). For $\lambda \in \C$, $\lambda \neq 1$, by means of the generating function
\begin{equation}
\label{eq:GF-AE}
	     g_{E}(x,\lambda,z)
  \eqdef \frac{2 e^{zx}}{\lambda e^z+1}
  =      \sum_{n=0}^\infty \E_n(x;\lambda)\, \frac{z^n}{n!},
\end{equation}
which converges for $|z| < |\log(-\lambda)|.$ 
For $\lambda = 1$, we have the classical Euler polynomials, $\E_{n}(x;1) = E_{n}(x)$.
For $\lambda = -1$, $g_{E}$ has a simple pole at $z = 0$ with residue~$-2$. We can include this case also by defining $\frac{1}{n!}\E_{n}(x;\lambda)$ to be, in general, the coefficient of 
$z^{n}$ in the Laurent expansion of $g_{E}$ around $z=0$, or by redefining $g_{E}$ to be the holomorphic part, i.e.\ adding $-\frac{2}{z}$ to it when $\lambda = -1$. The latter actually makes more sense as far as unifying results goes.

Formula (2.18) of~\cite{Luo-MC} or (37) of \cite{LuoSri} gives the relation
\begin{equation}
\label{eq:relAB-AE1}
	\E_{n}(x;\lambda) = \frac{2}{n+1} 
	                    \left(
	                        \B_{n+1}(x;\lambda) 
	                       - 2^{n+1}\B_{n+1}\left(\frac{z}{2};\lambda^{2}\right)
	                    \right)
\end{equation}
between the Apostol-Euler and Apostol-Bernoulli polynomials. However, it is easier to use 
the following relation to transfer results between them.

\begin{lemma}
\label{lem:relAB-AE2}
For all $\lambda \in \C$, we have
\begin{equation}
\label{eq:relAB-AE2}
	\E_{n}(x;\lambda) = - \frac{2}{n+1} \B_{n+1}(x;-\lambda).
\end{equation}
In particular, for $\lambda = 1$,
\begin{equation}
\label{eq:relAB-AE@-1}
	\B_{n}(x;-1) = - \frac{n}{2} E_{n-1}(x), \quad n \geq 1,
\end{equation}
and for $\lambda = -1$,
\begin{equation}
\label{eq:relAB-AE@1}
	\E_{n}(x;-1) = - \frac{2}{n+1} B_{n+1}(x), \quad n \geq 1.
\end{equation}
\end{lemma}

\begin{proof}
Writing $g_{B}$ for the generating function \eqref{eq:GF-AB} of the Apostol-Bernoulli polynomials, we easily see that $g_{B}(x,-\lambda,z) = -\frac{z}{2}g_{E}(x,\lambda,z)$, from which \eqref{eq:relAB-AE2} follows.
\end{proof}

\begin{remark}
The relation \eqref{eq:relAB-AE2} shows that the Apostol-Euler and Apostol-Bernoulli families are essentially the same. Apparently, \eqref{eq:relAB-AE2} is a new observation, although we should note that \eqref{eq:relAB-AE1} and \eqref{eq:relAB-AE2} are related by the following ``duplication formula'' for the Apostol-Bernoulli polynomials:
\begin{equation}
\label{eq:AB-mult@2}
	  2^{n} \B_{n}\left(\frac{x}{2}; \lambda^{2}\right)
	= \B_{n}(x;\lambda) + \B_{n}(x;-\lambda).
\end{equation}
Note also that by \eqref{eq:relAB-AE@-1}, Example \ref{ex:l=-1} actually describes the asymptotic behavior of the Euler polynomials.
\end{remark}

The poles of the generating function $g_{E}$ are the numbers $(2k + 1)\pi i - \log\lambda$ for $k \in \Z$. This is also the case for $\lambda = -1$, since $0$ is a pole of $g_{E}(x,-1,z)$.
However, the relations \eqref{eq:relAB-AE2} and \eqref{eq:relAB-AE@1} suggest that we exclude $0$ as a pole. This is consistent with redefining the generating function to be the holomorphic part of $g_{E}$ at~$0$. Thus we set
$$
	S_{E} = \begin{cases}
			\{ (2k + 1) \pi i - \log\lambda : k \in \Z \} & \text{if $\lambda \neq -1$},
			\\
			\{ 2 \pi i k : k \in \Z \} & \text{if $\lambda = -1$}.
			\end{cases}
$$
If we write $S_{B}$ for the pole set of $g_{B}$, and indicate the dependence on $\lambda$, this means we have the symmetric relation $S_{E}(\lambda) = S_{B}(-\lambda)$.

Now, via \eqref{eq:relAB-AE2} we immediately obtain the Fourier expansion of the Apostol-Euler polynomials as a special case of that of the Apostol-Bernoulli polynomials (see also~\cite{Luo-MC}):
\begin{equation}
\label{eq:FS-EB}
    \E_n(x;\lambda) 
  = \frac{2 \cdot n!}{\lambda^x} \sum_{k \in \Z} 
    \frac{e^{(2k+1)\pi i x}}{((2k+1)\pi i - \log \lambda)^{n+1}},
\end{equation}
valid for $0 \le x \le 1$ when $n \ge 1$ and for $0 < x < 1$ when $n = 0$.
This may be rewritten in a form completely analogous to \eqref{eq:FS-poles} as
\begin{equation}
\label{eq:FS-AE-poles}
	\frac{\E_n(x;\lambda)}{2 \cdot n!} = \sum_{a \in S_{E}} \frac{e^{ax}}{a^{n+1}}
\end{equation}
which provides further justification for redefining the generating function as we have indicated.

With \eqref{eq:FS-AE-poles} we now have the same results for the Apostol-Euler polynomials as we did for the Apostol-Bernoulli polynomials; indeed, \eqref{eq:relAB-AE2} says that they are just a special case. We can also think of the poles
$
	\widetilde{a}_k \eqdef (2k + 1) \pi i - \log\lambda
	           =      2\pi i \left( k - \tfrac{1}{2} - \xi \right)
$
with $\xi = \frac{\log\lambda}{2\pi i}$, as ``half-integer versions'' of the poles
$a_{k} = 2\pi i - \log\lambda$ of $g_{B}$. In any case, Lemma~\ref{lem:pole-ordering} still describes their ordering, and Theorem~\ref{thm:approx-AB} remains valid with \eqref{eq:FS-AE-poles} instead of~\eqref{eq:FS-poles}. The normalization of $\E_{n}(z;\lambda)$ analogous to
\eqref{def:ab-normalized} is given for $\lambda \neq 0,-1$ as
\begin{equation}
\label{def:ae-normalized}
\begin{aligned}
	 \varepsilon_{n}(z;\lambda) 
  &= -\beta_{n+1}(z;-\lambda)
   = (-1)^{n+1} \frac{(\log\lambda + \epsilon \pi i)^{n+1}}{2 \cdot n!} 
               \,e^{\epsilon \pi i z} \lambda^{z} \E_{n}(z;\lambda),
  \\
     \epsilon
  &= \begin{cases}
		+1 & \text{if $\lambda > 0$ or $\Im\lambda < 0$,}
		   \\
		-1 & \text{if $\lambda < 0$ or $\Im\lambda > 0$.}
	 \end{cases}               
\end{aligned}
\end{equation}
The sign $\epsilon$ is determined by the relation $\log(-\lambda) = \log(\lambda) + \epsilon \pi i$.  We limit ourselves to listing some of the analogous results formulated in terms of $\E_{n}$ rather than $\B_{n}$, without the specifics about the order of approximation, which are of course still valid. For example, the limit in Corollary~\ref{cor:0th-approx} is now
\begin{equation}
\label{eq:0th-approx-AE}
	   \lim_{n \to \infty} \varepsilon_{n}(z;\lambda)
	 = 1,
	 \quad \lambda \notin (0,\infty),\ \lambda \neq -1,
\end{equation}
while the oscillating case in Proposition~\ref{prop:l<0} is
\begin{equation}
\label{eq:AEl>0}
	     \lim_{n \to \infty} \left( \varepsilon_{n}(z;\lambda)
	   - \left(\frac{\log\lambda + \pi i}{\log\lambda - \pi i}\right)^{n+1} e^{2\pi i z} \right)
	 = 1,
	 \quad \lambda >0,
\end{equation}
and the limit of quotients in Proposition~\ref{prop:quot-AB} becomes
\begin{equation}
\label{eq:quot-AE}
	 \lim_{n \to \infty} \frac{\varepsilon_{n+1}(z;\lambda) - 1}{\varepsilon_{n}(z;\lambda) - 1}
	 = \frac{\log\lambda + \epsilon \pi i}{\log\lambda - \epsilon \pi i},
	 \quad
	 \lambda \notin (-\infty,0),
\end{equation}
all holding uniformly on compact subsets of~$\C$.




\begin{thebibliography}{9}

\bibitem{Ap} 
T. M. Apostol,
On the Lerch zeta function,
\textit{Pacific J. Math.}
\textbf{1} (1951), 161--167.

\bibitem{Bayad}
A. Bayad, 
Fourier expansions for Apostol-Bernoulli, Apostol-Euler and Apostol-Genocchi polynomials,
\textit{Math. Comp.}
to appear.

\bibitem{Dil} 
K. Dilcher,
Asymptotic behaviour of Bernoulli, Euler, and generalized Bernoulli polynomials,
\textit{J. Approx. Theory}
\textbf{49} (1987), 321--330.

\bibitem{Bate} 
A. Erd\'{e}lyi, W. Magnus, F. Oberhettinger and F. Tricomi,
\textit{Higher Transcendental Functions}, Volume I,
McGraw-Hill, New York, 1953.

\bibitem{NRV-JAT}
L. M. Navas, F. J. Ruiz and J. L. Varona,
The M\"obius inversion formula for Fourier series 
applied to Bernoulli and Euler polynomials,
\textit{J. Approx. Theory}
\textbf{163} (2011), 22--40.

\bibitem{Luo-Tai} 
Q.-M. Luo, 
Apostol-Euler polynomials of higher order and Gaussian hypergeometric functions, 
\textit{Taiwanese J. Math.} 
\textbf{10} (2006), 917--925.

\bibitem{Luo-MC} 
Q.-M. Luo, 
Fourier expansions and integral representations for the Apostol-Bernoulli and Apostol-Euler polynomials,
\textit{Math. Comp.} 
\textbf{78} (2009), 2193--2208. 

\bibitem{LuoSri}
Q.-M. Luo and H. M. Srivastava, 
Some relationships between the Apostol-Bernoulli and Apostol-Euler polynomials, 
\textit{Comput. Math. Appl.} 
\textbf{51} (2006), 631--642.

\end{thebibliography}
\end{document}